\documentclass[12pt]{amsart}
\usepackage{amsmath, amssymb, amsthm, bm, mathtools, enumitem, caption}
\usepackage{hyperref}

\usepackage[noabbrev,capitalize]{cleveref}
\usepackage{adjustbox}
\crefname{equation}{}{}
\usepackage{ytableau}
\usepackage{graphicx, tikz}
\usepackage[textheight=8.85in, textwidth=7in]{geometry}

\usepackage{microtype}

\newtheorem{theorem}{Theorem}[section]
\newtheorem{lemma}[theorem]{Lemma}
\newtheorem{corollary}[theorem]{Corollary}

\newtheorem*{conjecture*}{Conjecture}

\theoremstyle{definition}

\theoremstyle{remark}
\newtheorem*{remark}{Remark}
\newtheorem*{example}{Example}

\numberwithin{equation}{section}

\newcommand{\fm}{\mathfrak m}

\newcommand{\SL}{\mathrm{SL}}

\newcommand{\Q}{\mathbb Q}

\newcommand{\Z}{\mathbb Z}
\newcommand{\ord}{\mathrm ord}

\title[MacMahon's sums-of-divisors and allied $q$-series]{MacMahon's sums-of-divisors and allied $q$-series}
\date{\today}
\author{Tewodros Amdeberhan, Ken Ono \and Ajit Singh}
\address{Dept. of Mathematics, Tulane University, New Orleans, LA 70118}
\email{tamdeber@tulane.edu}
\address{Dept. of Mathematics, University of Virginia, Charlottesville, VA 22904}
\email{ko5wk@virginia.edu}
\email{ajit18@iitg.ac.in}

\begin{document}
\begin{abstract} Here we investigate the $q$-series
\begin{displaymath}
\begin{split}
\mathcal{U}_a(q)=\sum_{n=0}^{\infty} MO(a;n)q^n&:=\sum_{0< k_1<k_2<\cdots<k_a} \frac{q^{k_1+k_2+\cdots+k_a}}{(1-q^{k_1})^2(1-q^{k_2})^2\cdots(1-q^{k_a})^2},\\
\mathcal{U}_a^{\star}(q)=\sum_{n=0}^{\infty}M(a;n)q^n&:=\sum_{1\leq k_1\leq k_2\leq\cdots\leq k_a} \frac{q^{k_1+k_2+\cdots+k_a}}{(1-q^{k_1})^2(1-q^{k_2})^2\cdots(1-q^{k_a})^2}.
\end{split}
\end{displaymath}
MacMahon introduced the $\mathcal{U}_a(q)$ in his seminal work on partitions and divisor functions.  Recent works show that these series are
sums of quasimodular forms with weights $\leq 2a.$ We make this explicit by describing them in terms of Eisenstein series. We use these formulas to obtain explicit and general congruences for the coefficients $MO(a;n)$ and $M(a;n).$ Notably, we prove the conjecture of Amdeberhan-Andrews-Tauraso as the  $m=0$ special case of the infinite family of congruences 
$$
MO(11m+10; 11n+7)\equiv 0\pmod{11},
$$
and we prove that
$$
MO(17m+16; 17n+15)\equiv 0\pmod{17}.
$$
We obtain further formulae using the limiting behavior of these series.
For $n\leq a+\binom{a+1}2,$ we obtain a ``hook length'' formulae for $MO(a;n)$, and for $n\leq 2a$, we find that $M(a;n)=\binom{a+n-1}{n-a}+\binom{a+n-2}{n-a-1}.$  
\end{abstract}

\maketitle
\section{Introduction and Statement of Results}

At first glance, one might underestimate the value of the trivial observation that the number of partitions of an integer $n$ into identical parts  is also the number of divisors of $n$. This fact is a glimpse of a rich theory that relates integer partitions and divisor functions. 
Indeed, MacMahon's important paper \cite{MacMahon} is based on the idea of connecting partitions to divisor sums: partition 
of $n$ using $k_1$ repeated $s_1$ times, and $k_2$ repeated $s_2$ times, and so on through $k_a$ repeated $s_a$ times. Using this convention, he considered the sum of products of the \emph{multiplicities} $MO(a;n):=\sum s_1s_2\cdots s_a$ of size $n$ partitions, which has the generating function
\begin{align} \label{Mac1} 
\mathcal{U}_a(q):=\sum_{n\geq0}MO(a;n)\,q^n=\sum_{0< k_1<k_2<\cdots<k_a} \frac{q^{k_1+k_2+\cdots+k_a}}{(1-q^{k_1})^2(1-q^{k_2})^2\cdots(1-q^{k_a})^2}.
\end{align}
His work  \cite{MacMahon} is populated with beautiful divisor function identities, where $\sigma_{\nu}(n):=\sum_{d\mid n}d^{\nu},$  such as:
\begin{equation}\label{EisensteinIdentities}
\mathcal{U}_1(q)=\sum_{n\geq1}\sigma_1(n)q^n \ \ \ \ {\text {\rm and}}\ \ \ \ 
\mathcal{U}_2(q)=\sum_{n\geq1} \left(\frac{\sigma_1(n)}{8}-\frac{n\sigma_1(n)}{4}+\frac{\sigma_3(n)}{8}\right)q^n.
\end{equation}

To entice the reader, we offer the first few terms of $\mathcal{U}_1(q),\dots, \mathcal{U}_4(q)$:
\begin{displaymath}
\begin{split}
\mathcal{U}_1(q)&=q+3q^2+4q^3+7q^4+6q^5+12q^6+8q^7+\dots, \\
\mathcal{U}_2(q)&=q^3+3q^4+9q^5+15q^6+30q^7+45q^8+67q^9+\dots, \\
\mathcal{U}_3(q)&=q^6+3q^7+9q^8+22q^9+42q^{10}+81q^{11}+140q^{12}+\dots, \\
\mathcal{U}_4(q)&=q^{10}+3q^{11}+9q^{12}+22q^{13}+51q^{14}+97q^{15}+188q^{16}+\dots. \\
\end{split}
\end{displaymath}
The inequalities in definition (\ref{Mac1}) imply that
$q^{-\frac{a(a+1)}{2}}\cdot \mathcal{U}_a(q)=1+3q+\dots,$ while for
$a\geq 2,$ we have
$$
q^{-\frac{a(a+1)}{2}}\cdot \mathcal{U}_a(q)=1+3q+9q^2+\dots.
$$
Answering the natural question, we show that this sequence converges to a simple infinite product, which, by the  theory of Nekrasov-Okounkov \cite{NekOk},  gives  {\it hook length} formulae for many of the $MO(a;n).$ 

To make this precise, recall that a partition $\lambda=(\lambda_1,\dots,\lambda_{\ell})$ of $n$, denoted $\lambda\vdash n$, is a non-increasing sequence of positive integers that sum to $n$. Its {\it Young diagram} is the left-justified array of boxes where the row lengths are the parts. The {\it hook} $H(i,j)$ of the box in position $(i,j)$ consists of this box, together with those below it and those to its right.  Its {\it hook length} $h(i,j):=(\lambda_i-i)+(\lambda_j^{\prime}-j)+1$ is the number of such boxes, where $\lambda_j^{\prime}$  is the number of boxes in column $j$. Denote the multiset of hook lengths of $\lambda$ by $\mathcal{H}(\lambda)$.  Finally, we recall the ``exponential form'' of a partition $\lambda=(1^{m_1},2^{m_2},\dots,t^{m_t}),$  where  $m_i$ is the multiplicity of part $i.$   

\begin{example}
The exponential form of $\lambda=(4,4,2)$  is $\lambda=(1^0,2^1,3^0,4^2,5^0,6^0,7^0,8^0,9^0,10^0)\vdash 10$.
 Its Young diagram is given below, and shows that $\mathcal{H}(\lambda)=\{6,5,5, 4, 3, 2, 2, 2, 1, 1\}$.

$$\ytableausetup{centertableaux}
\begin{ytableau}
6 & 5& 3 &2 \\
5&4&2&1\\
2&1
\end{ytableau}
$$
\end{example}

We derive the following result using the work of Andrews-Rose \cite{Andrews-Rose} and Nekrasov-Okounkov \cite{NekOk}.

\begin{theorem}\label{UaLimit} The following are true:

\noindent
(i) If $a$ is a positive integer, then we have that
$$q^{-\frac{a(a+1)}{2}}\cdot \mathcal{U}_a(q)=\prod_{n\geq1}\frac{1}{(1-q^n)^3}+O(q^{a+1}).
$$

\noindent
(ii) If $n\leq a+\binom{a+1}2$, then we have that
$$
MO(a;n)=\sum_{\lambda\vdash n-a}\, \prod_{h\in\mathcal{H}(\lambda)}\left(\frac{2}{h^2}+1\right)=\sum_{\lambda\vdash n-a}\prod_{s=1}^{n-a}\binom{2+m_s}2.
$$
\end{theorem}

Inspired by the $\mathcal{U}_a(q),$ Amdeberhan-Andrews-Tauraso \cite{AAT}  initiated the study of the $q$-series
\begin{align} \label{Mac2} 
\mathcal{U}_a^{\star}(q):=\sum_{n\geq0}M(a;n)\,q^n=\sum_{1\leq k_1\leq k_2\leq\cdots\leq k_a} \frac{q^{k_1+k_2+\cdots+k_a}}{(1-q^{k_1})^2(1-q^{k_2})^2\cdots(1-q^{k_a})^2},
\end{align}
where the strict inequalities in (\ref{Mac1}) are replaced by weak inequalities. One easily sees that
 $$
 \mathcal{U}_a^*(q)=\sum_{n\geq0}M(a;n)q^n=q^a+(2a+1)q^{a+1}+\dots.
 $$
 To entice the reader, we offer the first few terms of $\mathcal{U}_1^*(q),\dots, \mathcal{U}_4^*(q)$:
\begin{displaymath}
\begin{split}
\mathcal{U}_1^{\star}(q)&=q+3q^2+4q^3+7q^4+6q^5+12q^6+\cdots, \\
\mathcal{U}_2^{\star}(q)&=q^2+5q^3+14q^4+29q^5+55q^6+86q^7+\cdots, \\
\mathcal{U}_3^{\star}(q)&=q^3+7q^4+27q^5+77q^6+181q^7+378q^8+\cdots, \\
\mathcal{U}_4^{\star}(q)&=q^4+9q^5+44q^6+156q^7+450q^8+1121q^9+\cdots.
\end{split}
\end{displaymath}

In analogy with Theorem~\ref{UaLimit}, we consider the limiting behavior of these series. These series converge to specializations of the generating function for the polynomials
$p_0(x):=1,
p_1(x):=2x+1,
p_2(x):=2x^2+3x,
p_3(x):=\frac{4}{3}x^3+4x^2+\frac{5}{3}x,\dots.$
For $n\geq 1,$ these polynomials are defined by
\begin{equation}
p_n(x):=\binom{2x+n-1}n+\binom{2x+n-2}{n-1}.
\end{equation}

As a companion to Theorem~\ref{UaLimit}, we obtain the following theorem.

\begin{theorem}\label{UaStarLimit} The following are true:

\noindent
(i) If $a$ is a positive integer, then we have that
$$
 q^{-a}\cdot \mathcal{U}^*_a(q)= \sum_{n=0}^{a}p_n(a)q^n+ O(q^{a+1}).
$$

\noindent
(ii) If $n\leq 2a$, then we have that
$M(a;n)=p_{n-a}(a).$
\end{theorem}

\begin{remark} The $\mathcal{U}_a(q)$ and $\mathcal{U}_a^{\star}(q)$ are
multiple $q$-zeta values. To make this precise,
we recall the $q$-notation $[k]_q:=\frac{1-q^k}{1-q}$ and the multiple $q$-zeta values  (for example, see \cite{Brindle}) 
\begin{align*}
\zeta_q(m_1,\dots,m_a):&=\sum_{0<k_1<\cdots<k_a}\frac{q^{(m_1-1)k_1+\cdots+(m_a-1)k_a}}{[k_1]_q^{m_1}\cdots[k_a]_q^{m_a}},\\
\zeta_q^{\star}(m_1,\dots,m_a):&=\sum_{1\leq k_1\leq\cdots\leq k_a}\frac{q^{(m_1-1)k_1+\cdots+(m_a-1)k_a}}{[k_1]_q^{m_1}\cdots[k_a]_q^{m_a}}.
\end{align*}
We have that 
$(1-q)^{2a}\cdot \mathcal{U}_a(q)=\zeta_q({2,\ldots,2})$ and $(1-q)^{2a}\cdot \mathcal{U}_a^{\star}(q)
=\zeta_q^{\star}({2,\ldots,2}).$
\end{remark}

As divisor functions arise as the coefficients of Eisenstein series, identities such as (\ref{EisensteinIdentities}) suggest a strong relationship between the
$\mathcal{U}_a(q)$ and quasimodular forms. This speculation was confirmed by Andrews-Rose. Indeed, they proved 
(see \cite[Cor. 4]{Andrews-Rose}) and \cite[Th. 1.12]{rose})  that each $\mathcal{U}_a(q)$ is a linear combination of quasimodular forms on $\SL_2(\mathbb{Z})$ with weights $\leq 2a$. 
Similarly,  Amdeberhan-Andrews-Tauraso \cite[Th. 6.1]{AAT}  proved that each $\mathcal{U}_a^{\star}(q)$ is a linear combination of quasimodular forms on $\SL_2(\mathbb{Z})$ with weights $\leq 2a$. 

Here we make this quasimodularity  explicit. In the case of $\mathcal{U}_a(q),$ we employ the standard generators of the graded ring of quasimodular forms: the quasimodular weight 2 Eisenstein series 
\begin{equation}
E_2(q):=1-24\sum_{n=1}^{\infty}\sigma_1(n)q^n,
\end{equation}
and the weight 4 and 6 modular Eisenstein series
\begin{equation}
E_4(q):=1+240\sum_{n=1}^{\infty}\sigma_3(n)q^n \qquad \text{and} \qquad
E_6(q):=1-504\sum_{n=1}^{\infty}\sigma_5(n)q^n.
\end{equation}
It is well known \cite{Zagier} that the ring of quasimodular forms is $\mathbb{C}[E_2,E_4,E_6],$ and so our goal is to 
 obtain formulas in terms of the monomials $E_2^{\alpha}(q)E_4^{\beta}(q)E_6^{\gamma}(q)$, where $\alpha, \beta$ and $\gamma$ are non-negative integers. 
 
Our formulas for $\mathcal{U}_a(q)$ use of the triple index sequence of rational numbers defined by the recursion
\begin{align} \label{coefficients} \nonumber
c(\alpha,\beta,\gamma)
&:=-\frac13(2\alpha+8\beta+12\gamma+1)\cdot c(\alpha-1,\beta,\gamma)+\frac23(\alpha+1)\cdot c(\alpha+1,\beta-1,\gamma) \\
&\qquad \qquad \ \ \ \ \  +\frac83(\beta+1)\cdot c(\alpha,\beta+1,\gamma-1)+4(\gamma+1)\cdot c(\alpha,\beta-2,\gamma+1),
\end{align}
where $\alpha, \beta, \gamma\geq0.$ To seed the recursion, we let $c(0, 0, 0):=1,$ and we let
$c(\alpha,\beta,\gamma):=0$ if any of the arguments are negative. 
Here we list the ``first few'' values: 
\begin{align*}
c(1,0,0)=-1, \  c(0,1,0)=-\frac{2}{3},\  c(0,0,1)=-\frac{16}9,  \   c(1,1,0)=\frac{14}3,  \  
c(1,0,1)=\frac{64}3,\dots.
\end{align*}
We also require constants for  the quasimodular summands sorted by weight.
For $0\leq t\leq a$, define
\begin{align} \label{depend-a-t}
w_t(a):=\frac{\binom{2a}a}{16^a(2a+1)} \sum_{0\leq \ell_1<\cdots<\ell_t<a}\prod_{j=1}^t\frac1{(2\ell_j+1)^2}.
\end{align}
In terms of $w_t(a)$ and the numbers $c(\alpha, \beta, \gamma)$, we have the following explicit
formulae for $\mathcal{U}_a(q).$

\begin{theorem}\label{ExplicitExpressionUa} If $a$ is a non-negative integer, then we have that
$$
\mathcal{U}_a(q)= \sum_{t=0}^a w_t(a)
\sum_{\substack {\alpha, \beta, \gamma\geq 0\\ \alpha+2\beta+3\gamma=t}}
c(\alpha,\beta,\gamma) E_2(q)^{\alpha} E_4(q)^{\beta} E_6(q)^{\gamma}.
$$
\end{theorem}

\begin{example} For $a=3$, Theorem~\ref{ExplicitExpressionUa} gives
\begin{displaymath}
\begin{split}
\mathcal{U}_3(q)&=\frac{5}{7168}-\frac{37E_2(q)}{46080}+\frac{5E_2(q)^2}{27648}-\frac{E_4(q)}{13824}-\frac{E_2(q)^3}{82944}+\frac{E_2(q)E_4(q)}{69120}-\frac{E_6(q)}{181440}.
\end{split}
\end{displaymath}
\end{example}

We  turn to the $\mathcal{U}_a^*(q).$ Instead of using $E_2(q), E_4(q),$ and $E_6(q),$ we use all of the Eisenstein series
\begin{equation} \label{EandSigma}
E_{2k}(q):=1-\frac{4k}{B_{2k}}\sum_{n=1}^{\infty}\sigma_{2k-1}(n)q^n,
\end{equation}
where $B_k$ is the usual $k$th Bernoulli number. Namely, we let $\mathbb{E}_0(q):=1,$ and for positive  $t$ we define
\begin{align} \label{E-star}
\mathbb{E}_{2t}^{\star}(q)
:=\sum_{(1^{m_1},\dots,t^{m_t})\vdash t}\,
\prod_{j=1}^t\frac1{m_j!}\left(-\frac{B_{2j}\,E_{2j}(q)}{(2j)\cdot (2j!)}\right)^{m_j}.
\end{align}
We require constants for the summands sorted by weight. We let $w_0^{\star}(0):=1$, and for $a>0,$ we let 
\begin{equation}
w_0^{\star}(a):=\sum_{i=1}^a\frac{(-1)^{i-1}\binom{2i}i}{16^i(2i+1)}\,w_0^{\star}(a-i).
\end{equation}
For $1\leq t\leq a$, we define
\begin{equation}
w_t^{\star}(a):=(-1)^{a+t-1}4^t(2t+1)!\,w_{t-1}(a-1).
\end{equation}
With this notation, we obtain the following explicit expressions for $\mathcal{U}_a^*(q)$.

\begin{theorem}\label{ExplicitExpressionUaStar} If $a$ is a non-negative integer, then we have that
$$
\mathcal{U}_a^{\star}(q)=\sum_{t=0}^a w_t^{\star}(a) \cdot \mathbb{E}_{2t}^{\star}(q).
$$
\end{theorem}
\begin{example} For $a=5$, Theorem~\ref{ExplicitExpressionUaStar} gives
\begin{displaymath}
\begin{split}
\mathcal{U}_5^*(q)&=\frac{1295803}{12262440960}+\frac{35}{294912}\mathbb{E}_2^{\star}(q)-\frac{3229}{967680}\mathbb{E}_4^{\star}(q)+\frac{47}{1152}\mathbb{E}_6^{\star}(q)
-\frac7{24}\mathbb{E}_8^{\star}(q)+\mathbb{E}_{10}^{\star}(q).
\end{split}
\end{displaymath}
\end{example}

The coefficients of $\mathcal{U}_a(q)$ and $U_a^{\star}(q)$ satisfy surprising congruences. Amdeberhan-Andrews-Tauraso \cite{AAT} discovered some 
congruences that are reminiscent of Ramanujan's partition congruences, such as
\begin{displaymath}
\begin{split}
MO(2;5n+2)\equiv 0\pmod 5 \ \ \ \ \ {\text {\rm and}}\ \ \ \ 
MO(3;7n+3)\equiv MO(3;7n+5)\equiv 0\pmod 7.
\end{split}
\end{displaymath}
Moreover, they conjectured (see Conjecture 9.1 of \cite{AAT}) that
\begin{equation}\label{AATConjecture}
MO(10; 11n+7)\equiv 0\pmod{11}.
\end{equation}

\begin{theorem}\label{Mod11}
For every non-negative integer $n$, we have that
$$
MO(10;11n+7)\equiv 0\pmod{11}.
$$
\end{theorem}
We offer two proofs of this result. The first proof uses
the explicit description of $\mathcal{U}_{10}(q)$ provided by Theorem~\ref{ExplicitExpressionUa},  which allows us to employ the ``theory of modular forms mod $p$''. This proof illustrates an algorithm that reduces the proof of all conjectured congruences of the form
$$
MO(a;pn+r)\equiv 0\pmod{p} \ \ \ \ {\text {\rm and}}\ \ \ \ M(a;pn+r)\equiv 0\pmod p
$$
to finitely many steps. Theorem~\ref{Mod11} requires computing at most 20 terms of five auxiliary $q$-series.

The second proof is a special case of one of three new infinite families of congruences.
\begin{theorem}\label{Gordon} The following are true:

\noindent
(i) For every pair of non-negative integers $n$ and $m,$ we have that
$$
MO(3m+2;3n+1)\equiv MO(3m+2;3n+2)\equiv  0\pmod 3.
$$

\noindent
(ii) For every pair of non-negative integers $n$ and $m,$ we have
\begin{displaymath}
MO(11m+10; 11n+7)\equiv  0\pmod{11}.
\end{displaymath}

\noindent
(iii) For every pair of non-negative integers $n$ and $m,$ we have
\begin{displaymath}
MO(17m+16; 17n+15)\equiv  0\pmod{17}.
\end{displaymath}

\end{theorem}

\noindent
Computer searches for congruences suggest that such congruences are rare,
thereby underscoring the significance of Theorem~\ref{Mod11}.
However, it turns out that congruences are both rare and ubiquitous.

\begin{theorem}\label{Lacunarity}
For positive integers $a$ and $m$, the following are true:

\noindent
(i) There are infinitely many non-nested arithmetic progressions $tn+r$ {\rm(}resp. $t^*n+r^*${\rm )} for which
\begin{displaymath}
\begin{split}
M(a;tn+r)&\equiv 0\pmod m,\\
MO(a;t^*n+r^*)&\equiv 0\pmod m.
\end{split}
\end{displaymath}

\noindent
(ii) There are infinitely many non-nested arithmetic progressions $tn+r$ for which
$$
M(a;tn+r)\equiv MO(a;tn+r)\equiv 0\pmod m.
$$

\noindent
(iii) There exists a positive real number $\alpha(a,m)>0$ for which
\begin{displaymath}
\begin{split}
\# \{ n\leq X \ : \ M(a;n) \not \equiv 0\pmod m\} &= O\left(X/\log^{\alpha(a,m)} X \right)\\
\# \{ n\leq X \ : \ MO(a;n) \not \equiv 0\pmod m\} &= O\left(X/\log^{\alpha(a,m) X}\right).
\end{split}
\end{displaymath}
In other words, the values $M(a;n)$ and $MO(a;n)$ are almost always multiples of any integer $m.$
\end{theorem}

To conclude, we offer infinite families of congruences when $a\in\{2,3,4,5\}$. For convenience, we let
\begin{align}
N_{a}:=\begin{cases}
2^3 & \mbox{if $a=2$},\\
2^{7}3\cdot 5 & \mbox{if $a=3$},\\
2^{10}3^{3}\cdot 5\cdot 7 & \mbox{if $a=4$},\\
2^{15}3^{3} 5^2\cdot 7 & \mbox{if $a=5$}.\nonumber
\end{cases}
\end{align}	

\begin{corollary} \label{natural}
If  $a\in \{2, 3, 4, 5\}, $ then the following are true:

\noindent
(i)  If  $\ell \in \{2, 3, 5, 7\}$ and $p\equiv -1\pmod{\ell^{\ord_\ell(N_{a})+1}}$ is prime, then for every $n$ coprime to $p$ we have
$$
MO(a;pn)\equiv 0\pmod{\ell}.
$$

\noindent
(ii) If $\ell\geq 11$ is prime  and $p\equiv-1\pmod{\ell}$, then for every integer $n$ coprime to $p$ we have
$$
MO(a;pn)\equiv 0\pmod{\ell}.
$$	
\end{corollary}

\begin{example} The following congruence is an example of Corollary~\ref{natural} $(i):$
\begin{displaymath}
\begin{split}
MO(2,19^2n+19)&\equiv MO(2,19^2n+38)\equiv MO(2,19^2n+57)\equiv MO(2,19^2n+76) \equiv 0 \pmod{5},\\
\end{split}
\end{displaymath}
As an example of Corollary~\ref{natural} $(ii)$, for  $1\leq t\leq18$, we have
\begin{displaymath}
\begin{split}
MO(a,37^2n+37t)&\equiv0\pmod{19}.\\
\end{split}
\end{displaymath}
\end{example}

\begin{remark}
	Most of the congruences in Theorem~\ref{Lacunarity} do not belong to infinite families such as those in Corollary~\ref{natural}. For instance, if $p\in\{67,101,271,373\}$, then for every non-negative integer $n$
	we have 
	$$M(6;pn)\equiv MO(6;pn)\equiv 0\pmod{17}.$$ 
The coefficients of the expansion of $\mathcal{U}_{6}(q)$ provided by Theorem~\ref{ExplicitExpressionUa} are units modulo 17, and so these congruences follow from the fact that
all of the monomials  $E_2(q)^{\alpha} E_4(q)^{\beta} E_6(q)^{\gamma},$ with $\alpha, \beta, \gamma\geq 0$ and 
$\alpha+2\beta+3\gamma\leq 6$, are annihilated modulo 17 by the Hecke operators $T_p$ for $p\in\{67,101,271,373\}$.
\end{remark}

This paper is organized as follows.  In Section~\ref{Section2}, we recall the Nekrasov-Okounkov hook formulae and relevant  results of Andrews-Rose and Amdeberhan-Andrews-Tauraso, which we then employ to prove
 Theorems~\ref{UaLimit} and~\ref{UaStarLimit} on the limiting behavior of $U_a(q)$ and $U_a^{\star}(q)$. In Section~\ref{Section3} we recall pertinent facts about symmetric functions, as well as results on the quasimodularity of $U_a(q)$, which we then use to  prove Theorems~\ref{ExplicitExpressionUa} and ~\ref{ExplicitExpressionUaStar}.  Finally, in Section 4 we prove Theorems~\ref{Mod11} and \ref{Gordon}, and in Section~\ref{Section5} we prove Theorem~\ref{Lacunarity} using modularity. 
 
\section*{Acknowledgements}
\noindent
 The second author thanks the Thomas Jefferson Fund and the NSF
(DMS-2002265 and DMS-2055118). The third author is grateful for the support of a Fulbright Nehru Postdoctoral Fellowship.

\section{Proofs of Theorems~\ref{UaLimit} and \ref{UaStarLimit}}\label{Section2}

Here we prove Theorems~\ref{UaLimit} and ~\ref{UaStarLimit} using earlier work of Nekrasov-Okounkov and Andrews-Rose.

\subsection{Proof of Theorem~\ref{UaLimit}}

We require a beautiful identity of Andrews-Rose for $\mathcal{U}_a(q)$.

\begin{lemma}\cite[Cor. 2]{Andrews-Rose} \label{newgen_for_U} If $a$ is a positive integer, then as formal power series we have that
$$\mathcal{U}_a(q)\cdot \prod_{n\geq1}(1-q^n)^3=\frac{(-1)^a}{(2a+1)!}\,\sum_{n\geq0}(-1)^n(2n+1)\frac{(n+a)!}{(n-a)!}\, q^{\frac{n(n+1)}2}.$$
\end{lemma}

\noindent
We also require the celebrated Nekrasov-Okounkov hook length identity (see (6.12) on page 569 of \cite{NekOk}; see also Th.~1.3 of \cite{Han}).

\begin{theorem}\label{NOIdentity} As a formal power series, we have
$$\prod_{j\geq1}\frac1{(1-q^j)^{z+1}}=\sum_{m\geq0}q^m\, \sum_{\lambda\vdash m}
\prod_{h\in\mathcal{H}(\lambda)}\left(\frac{z}{h^2}+1\right).$$
\end{theorem}

\begin{proof}[Proof of Theorem~\ref{UaLimit}]
Thanks to Lemma~\ref{newgen_for_U} for $\mathcal{U}_a(q),$ we find that
\begin{align*} q^{-\binom{a+1}2}\cdot \mathcal{U}_a(q)\cdot \prod_{n\geq1}(1-q^n)^3
&=\sum_{j\geq0}(-1)^j\cdot \frac{2j+2a+1}{2a+1}\cdot \binom{j+2a}j\,q^{aj+\binom{j+1}2} \\
&=1-(2a+3)q^{a+1}+ (a+1)(2a+5)q^{2a+3}+\cdots \,\, . 
\end{align*}
Claim $(i)$ follows immediately.

The first formula in $(ii)$ follows by letting $z=2$ in Theorem ~\ref{NOIdentity}, giving
$$\prod_{n\geq1}\frac1{(1-q^n)^3}=\sum_{m\geq0}q^m\, \sum_{\lambda\vdash m} \,\prod_{h\in\mathcal{H}(\lambda)}\left(\frac{2}{h^2}+1\right),$$
while the other claim arises from the interpretation of the $q$-product in terms of  3-colored partitions.
\end{proof}

\subsection{Proof of Theorem~\ref{UaStarLimit}}

Amdeberhan-Andrews-Tauraso express $\mathcal{U}_a^{\star}(q)$ as  a single
sum.

\begin{lemma}\label{AAT-prop}{\text {\rm \cite[Prop. 4.1]{AAT}}} We have the identity
$$
\mathcal{U}_a^*(q)=\sum_{k\geq1}(-1)^{k-1}\,\frac{(1+q^k)\,q^{\binom{k}2+ak}}{(1-q^k)^{2a}}.$$
\end{lemma}
\begin{proof}[Proof of Theorem~\ref{UaStarLimit}] The expansion $(1-q^k)^{-2a}=\sum_{m\geq0}\binom{2a+m-1}m\,q^{km}$ and Lemma~\ref{AAT-prop} imply that
\begin{align*}
q^{-a}\cdot \mathcal{U}_a^{\star}(q)
&=\sum_{k\geq1}(-1)^{k-1}\,\frac{q^{\binom{k}2+a(k-1)}}{(1-q^k)^{2a}} + \sum_{k\geq1}(-1)^{k-1}\,\frac{q^{\binom{k+1}2+a(k-1)}}{(1-q^k)^{2a}} \\
&=\sum_{k\geq1}\sum_{m\geq0} (-1)^{k-1}\binom{2a+m-1}m\,q^{km+\binom{k}2+a(k-1)}+
\sum_{k\geq1}\sum_{m\geq0} \binom{2a+m-1}m\, q^{km+\binom{k+1}2+a(k-1)}.
\end{align*}
We compare coefficients of $q^n$ for $n\leq a$. Namely, in the double sums we require $km+\binom{k}2+a(k-1)\leq a$ and $km+\binom{k+1}2+a(k-1)\leq a$. The former results in $k=1, m=n$ and the latter forces $k=1, m=n-1$. Consequently, if we let $q^{-a}\cdot \mathcal{U}_a^{\star}(q)=:\sum_{n\geq0}p_n(a)\,q^n$, then we find that
$$p_n(a)=\binom{2a+n-1}n+\binom{2a+n-2}{n-1}.$$
\end{proof}

\section{Proof of Theorems~\ref{ExplicitExpressionUa} and \ref{ExplicitExpressionUaStar}}\label{Section3}

Here we prove the explicit descriptions of $\mathcal{U}_a(q)$ and $\mathcal{U}_a^*(q)$ in terms of Eisenstein series.

\subsection{Nuts and Bolts}

We make use of the differential operator $\Theta:=q\frac{d}{dq}$, which acts by
\begin{equation}
\Theta\left(\sum a(n)q^n\right):=\sum na(n)q^n.
\end{equation}
Ramanujan famously obtained the following formulas \cite[p. 181]{Rama} for the action of $\Theta$:
\begin{align} \label{rama}
\Theta (E_2(q))&=\frac{E_2^2(q)-E_4(q)}{12}, \qquad \Theta(E_4(q))=\frac{E_2(q)E_4(q)-E_6(q)}3,  \\
 \Theta (E_6(q))&=\frac{E_2(q)E_6(q)-E_4^2(q)}2. \nonumber
\end{align}

The $q$-series $\mathcal{U}_a(q)$ and $\mathcal{U}_a^{\star}(q)$ satisfy the following convenient convolution (see \cite[p. 13]{AAT}).

\begin{lemma} \label{ele-hom} If $a$ is a positive integer, then we have that
$$\sum_{i=0}^a (-1)^i\cdot \mathcal{U}_i(q)\cdot \mathcal{U}_{a-i}^{\star}(q)=0.$$
\end{lemma}

Recall the Dedekind eta-function
$\eta(q)=q^{\frac1{24}}\prod_{m=1}^{\infty}(1-q^m).$
 The following result of Rose \cite[Th. 1.12]{rose}  describes the structural framework of  $\mathcal{U}_a(q)$ in terms of iterated derivatives of $\eta(q)^3.$

\begin{theorem} \label{weight} Each $\mathcal{U}_a(q)$ is a finite sum of quasimodular forms with weight $\leq 2a$ on $\SL_2(\mathbb{Z})$. Moreover, the weight $2t$ summand is a (possibly zero) scalar multiple of 
$$
2^t\cdot\frac{\Theta^t\left(\eta(q)^3\right)}{\eta(q)^3}.
$$
\end{theorem}

Our next result expresses these $q$-series  as a linear combination of monomials $E_2(q)^{\alpha} E_4(q)^{\beta} E_6(q)^{\gamma}$.

\begin{lemma} \label{preThm1.1} If $t$ is a positive integer, then we have that 
$$
(-8)^t\cdot\frac{\Theta^t\left(\eta(q)^3\right)}{\eta(q)^3}
=\sum_{\substack {\alpha, \beta, \gamma\geq 0\\ \alpha+2\beta+3\gamma=t}}
c(\alpha,\beta,\gamma) \cdot E_2(q)^{\alpha} E_4(q)^{\beta} E_6(q)^{\gamma}$$
where the coefficients $c(\alpha.\beta,\gamma)$ are defined by \eqref{coefficients}.
\end{lemma}

\begin{proof}
For convenience, we let $\psi(q):=\eta(q)^3.$  We calculate $\frac{\Theta^t(\psi(q))}{\psi(q)}$ by inducting on $t$. First,
it is easy to check $\Theta (\psi(q))=\frac18 \psi(q)E_2(q)$. Theorem \ref{weight} implies the existence of numbers $\widetilde{c}(\alpha,\beta,\gamma)$ for which
$$\frac{\Theta^t(\psi(q))}{\psi(q)}=\sum_{\substack {\alpha, \beta, \gamma\geq0 \\ \alpha+2\beta+3\gamma=t}} 
\widetilde{c}(\alpha,\beta,\gamma)
\cdot E_2^{\alpha}(q)E_4^{\beta}(q)E_6^{\gamma}(q).$$
This comprises of all weight $2t$ quasimodular summands in $\mathcal{U}_a(q)$. One more derivative $\Theta=q\frac{d}{dq}$ turns the last equation into (for brevity, we write $\widetilde{c}$ in place of $\widetilde{c}(\alpha,\beta,\gamma)$)
\begin{align*}
\Theta^{t+1}(\psi(q))
&=\Theta (\psi(q))\cdot\left(\sum_{\alpha,\beta,\gamma}\widetilde{c} \cdot E_2^{\alpha}(q)E_4^{\beta}(q)E_6^{\gamma}(q)\right)
+\psi(q)\cdot \sum_{\alpha, \beta, \gamma}\widetilde{c}\cdot \Theta(E_2^{\alpha}(q)E_4^{\beta}(q)E_6^{\gamma}(q)).
\end{align*}
On the other hand, Ramanujan's identities \eqref{rama} imply that
\begin{align*}
\Theta(E_2^{\alpha}E_4^{\beta}E_6^{\gamma})
&=\left(\frac{\alpha}{12}+\frac{\beta}3+\frac{\gamma}2\right)E_2^{\alpha+1}E_4^{\beta}E_6^{\gamma}-\frac{\alpha}{12}E_2^{\alpha-1}E_4^{\beta+1}E_6^{\gamma}  
- \frac{\beta}3E_2^{\alpha}E_4^{\beta-1}E_6^{\gamma+1}  - \frac{\gamma}2E_2^{\alpha}E_4^{\beta+2}E_6^{\gamma-1}.
\end{align*}
We find that the homogeneous weight $2t+2$ form satisfies
\begin{align*}
\frac{\Theta^{t+1}(\psi(q))}{\psi(q)}
&=\sum_{\substack {\alpha, \beta, \gamma\geq0 \\ \alpha+2\beta+3\gamma=t}} 
\left(\frac{\alpha}{12}+\frac{\beta}3+\frac{\gamma}2+\frac18\right)
\widetilde{c}\cdot E_2^{\alpha+1}E_4^{\beta}E_6^{\gamma} 
-\sum_{\alpha,\beta,\gamma}\frac{\alpha}{12} \,\widetilde{c}\cdot E_2^{\alpha-1}E_4^{\beta+1}E_6^{\gamma}  \\
& \qquad - \sum_{\alpha,\beta,\gamma}\frac{\beta}3 \, \widetilde{c}\cdot E_2^{\alpha}E_4^{\beta-1}E_6^{\gamma+1}  
- \sum_{\alpha,\beta,\gamma}\frac{\gamma}2 \, \widetilde{c}\cdot E_2^{\alpha}E_4^{\beta+2}E_6^{\gamma-1}.
\end{align*}
By comparing the coefficients of $E_2^{\alpha}E_4^{\beta}E_6^{\gamma}$ on both sides of the equation above,
we obtain the recursion  (with $\widetilde{c}(\alpha,\beta,\gamma)=\delta_{(0,0,0)}(\alpha,\beta,\gamma)$, a Dirac delta boundary conditions)
\begin{align*}
\widetilde{c}(\alpha,\beta,\gamma)
&=\left(\frac{\alpha}{12}+\frac{\beta}3+\frac{\gamma}2+\frac1{24}\right)\widetilde{c}(\alpha-1,\beta,\gamma)
-\frac{\alpha+1}{12}\cdot \widetilde{c}(\alpha+1,\beta-1,\gamma) \\
&\qquad -\frac{\beta+1}3\cdot \widetilde{c}(\alpha,\beta+1,\gamma-1) - \frac{\gamma+1}2\cdot 
\widetilde{c}(\alpha,\beta-2,\gamma+1).
\end{align*}

To determine the exact weight $2t$ term (independent of $a$), we take into account the
  factor of $(-8)^{\alpha+2\beta+3\gamma}$ to determine $c(\alpha,\beta,\gamma):=(-8)^{\alpha+2\beta+3\gamma}\cdot \widetilde{c}(\alpha,\beta,\gamma)$. As a result, we obtain the desired
  \begin{align*}
c(\alpha,\beta,\gamma)
&=-\frac13\left(2\alpha+8\beta+12\gamma+1\right)\cdot c(\alpha-1,\beta,\gamma)
+\frac23(\alpha+1) \cdot c(\alpha+1,\beta-1,\gamma) \\
&\qquad +\frac83(\beta+1) \cdot c(\alpha,\beta+1,\gamma-1) +4(\gamma+1) \cdot c(\alpha,\beta-2,\gamma+1).
\end{align*}
\end{proof}

\subsection{Proof of Theorem~\ref{ExplicitExpressionUa}}
We let $\mathcal{E}_t(q):= (-8)^t\cdot \frac{\Theta^t(\psi)}{\psi},$ and we define
$$\mathbb{E}_{2t}(q):=\sum_{(1^{m_1},\dots,t^{m_t})\vdash t}\,\prod_{j=1}^t\frac1{m_j!}\left(\frac{B_{2j}\,E_{2j}(q)}{(2j)\cdot (2j)!}\right)^{m_j}.$$
By inspection, we see that $\mathbb{E}_{2t}(q)$ has weight $2t.$
We claim that 
\begin{align} \label{EandEta} 
\mathbb{E}_{2t}(q)=\frac{(-1)^t}{4^t(2t+1)!}\cdot \mathcal{E}_t(q).
\end{align}
Let $\mathbf{S}_r(q):=\sum_{m\geq1}\frac{m^rq^m}{1-q^m}=\sum_{n\geq1}\sigma_r(n)q^n$. By expanding $\sum_{j,k\geq1}\frac{q^{kj}\cos(2kx)}k$ in two different ways, we find that it equals both of these
\begin{align} \label{proSigma}
\prod_{j\geq1} \left[1+\frac{4(\sin^2x)q^j}{(1-q^j)^2}\right]
&=\exp\left(-2\sum_{r\geq1}\frac{\mathbf{S}_{2r-1}(q)}{(2r)!} (-4x^2)^r\right).
\end{align}
Using the identity \cite[p. 13]{AAT}, we obtain
\begin{equation}\label{GoodFormula}
\prod_{k\geq1}\left (1+\frac{4q^k\sin^2x}{(1-q^k)^2}\right)=\sum_{a\geq0}4^a\mathcal{U}_a(q)(\sin x)^{2a},
\end{equation} 
and the Jacobi Triple Product then implies that
\begin{align} \label{proEta}
\sin x\prod_{k\geq1}\left(1+\frac{4q^k\sin^2x}{(1-q^k)^2}\right)
&=\frac{e^{ix}-e^{-ix}}{2i}\prod_{j\geq1}\frac{(1-q^je^{2ix})(1-q^je^{-2ix})}{(1-q^j)^2}\nonumber\\
&=\frac1{2i\cdot\psi(q)}\sum_{j\in\mathbb{Z}}(-1)^jq^{\binom{j+1}2}e^{(2n+1)ix} \nonumber \\
&=\frac1{\psi(q)}\sum_{t\geq0}(-1)^t\frac{x^{2t+1}}{(2t+1)!}\sum_{n\geq0}(-1)^n(2n+1)^{2t+1}q^{\binom{n+1}2}\nonumber\\
&=\sum_{t\geq0}\mathcal{E}_t(q)\frac{x^{2t+1}}{(2t+1)!}.
\end{align}
Using \eqref{EandSigma} and the generating function for P\'olya's \emph{cycle index formula} \cite[(1,5)]{Polya}, 
we obtain
\begin{align} \label{SigmaPolya}
\sin x\cdot \exp\left(-2\sum_{r\geq1}\frac{\mathbf{S}_{2r-1}(q)}{(2r)!} \,(-4x^2)^{2r}\right) 
&=\sin x\cdot\frac{x}{\sin x}\cdot\sum_{t\geq0} \left(\sum_{\lambda\vdash t}\prod_{j=1}^t\frac1{m_j!}
\left(\frac{B_{2j}\cdot E_{2s}(q)}{(2j)\cdot (2j!}\right)^{m_j}\right)(-4x^2)^t.
\end{align}
Combining \eqref{proSigma}, \eqref{proEta}, \eqref{SigmaPolya} and then comparing the coefficients of $x^{2t+1},$ we confirm \eqref{EandEta}. 

To the complete the proof, it suffices to determine the constants $b_t(a)$ for which
\begin{align} \label{structure}
\mathcal{U}_a(q)&=\sum_{t=0}^ab_t(a)\cdot \mathbb{E}_{2t}(q). 
\end{align}
It is convenient to recall the Andrews-Rose recursion \cite[Cor. 3]{Andrews-Rose}
\begin{align} \label{ARrecur}
\mathcal{U}_a(q)&=\frac1{2a(2a+1)}\left[(6\mathcal{U}_1(q)+a(a-1))\mathcal{U}_{a-1}(q)-2\Theta (\mathcal{U}_{a-1}(q))\right].
\end{align}
The structure of equation \eqref{structure} is preserved by \eqref{ARrecur} because of the identity
$$\Theta(\mathbb{E}_{2t-2})=t(2t+1)\mathbb{E}_{2t}-3\mathbb{E}_2\mathbb{E}_{2t-2}.$$
It is straightforward to see that
$$b_t(a)=\frac1{8a(2a+1)}\left[(2a-1)^2\cdot b_t(a-1)-8t(2t+1)\cdot b_{t-1}(a-1)\right],$$
with initial boundary conditions $b_0(0)=1$ and $b_t(a)=0$ when $t<0$ or $t>a$.
Finally, one checks that  $(-4)^t(2t+1)!\, w_t(a)$ satisfies this recurrence, thereby completing the proof of the theorem.

\subsection{Proof of Theorem~\ref{ExplicitExpressionUaStar}}
By reciprocating (\ref{GoodFormula}), we have
$$\sum_{n\geq0}(-4)^a\,\mathcal{U}_a^{\star}(q)(\sin x)^{2a}
=\prod_{k\geq1}\frac1{1+\frac{4q^k\sin^2x}{(1-q^k)^2}}.
$$
In analogy with the previous formula for $\mathbb{E}_{2t}(q)$ involving the $\mathcal{U}_a(q)$, we use \eqref{EandEta} to obtain an identity for  $\mathcal{U}_a^{\star}(q)$ with $\mathbb{E}_{2t}^*(q)$ (see \eqref{E-star}). 
Arguing as in the proof of Theorem~\ref{ExplicitExpressionUa} with Lemma~\ref{ele-hom}, 
we get 
$$\mathcal{U}_a^{\star}(q)=\sum_{t=0}^a w_t^{\star}(a)\cdot \mathbb{E}_{2t}^{\star}(q).$$

\section{Proof of Theorems~\ref{Mod11} and~\ref{Gordon}}\label{Section4}

Here we prove Theorem~\ref{Mod11} using Serre's theory of modular forms modulo primes $p$ (see  \cite[Section 2.8]{CBMS}, or \cite{SwD}) and a well-known criterion of Sturm that determines congruences between modular forms. In the sequel, we tacitly assume that $q:=e^{2\pi i z},$ the uniformizer for the point at infinity. We also prove Theorem~\ref{Gordon} by combining work of Andrews-Rose with  a classical result of Gordon, together with other allied observations.

\subsection{Modular forms modulo $p$}

We recall some facts from the theory of modular forms mod $p$. The key tool in the proof of Theorem~\ref{Mod11} is the following theorem of Sturm (see \cite{St} or p. 40 of \cite{CBMS}). 

\begin{theorem}\label{Sturm}
	Let $p$ be a prime. If $f(z)=\sum_{n\geq 0} a(n)q^n$ and $g(z)=\sum_{n\geq 0} b(n)q^n$ are modular forms of weight $k$ on  $\SL_2(\Z)$ with integer coefficients, then $f(z)\equiv g(z)\pmod p$ if and only if $a(n)\equiv b(n) \pmod p$ for all $n\leq k/12$. 
\end{theorem}

\noindent
We shall make use of derivatives of modular forms. Although differentiation does not preserve modularity, it does preserve modular forms modulo $p$ (for example, see \cite{Serre1}).

\begin{lemma}\label{lem1}
	If $f(z)=\sum_{n\geq 0} a(n)q^n\in M_k \cap \Z[[q]]$, then there is a modular form $g(z)=\sum_{n\geq 0} b(n)q^n \in M_{k+p+1}\cap \Z[[q]]$ for which
	$$
	g\equiv \Theta(f):=\sum_{n\geq 0}na(n)q^n\pmod p.
	$$
\end{lemma}

\subsection{Proof of Theorem~\ref{Mod11}} 
We let
$\mathcal{U}_{10}(q)=F_0(q)+F_2(q)+F_4(q)+F_6(q)+F_8(q),$
where each  $F_{2i}(q)$  is a sum of $E_2^{\alpha}E_4^{\beta}E_6^{\gamma}$ (suppressing the $q$), where $2\alpha+4\beta+6\gamma\equiv 2i\pmod{10}.$ 
Theorem \ref{ExplicitExpressionUa} then gives
 \begin{displaymath}
\begin{split}
F_0(q)=&\frac{46189}{5772436045824}-\frac{2008213E_4E_6}{4271802792542208000}
+\dots +\frac{E_2^{10}}{230078188847156428800}, \\
F_2(q)=&-\frac{25587296781661E_2}{2645567198945303592960}-\frac{604841E_6^2}{48057781416099840000}
+\dots+\frac{7862933E_2^6}{63910608013099008000}, \\
F_4(q)=&-\frac{79923511502753E_4}{67133754108574433280000}+\frac{79923511502753E_2^2}{26853501643429773312000} 
+\dots-\frac{16333E_2^7}{4473742560916930560}, \\
F_6(q)=&-\frac{70726885883E_6}{333200617818292224000}+\frac{70726885883E_2E_4}{126933568692682752000}-
\dots +\frac{1819E_2^8}{25564243205239603200}, \\
F_8(q)=&-\frac{316100258731E_4^2}{20732482886471516160000}+\frac{316100258731E_2E_6}{3887340541213409280000}-
 \dots -\frac{19E_2^9}{23007818884715642880}.
\end{split}
\end{displaymath}
Each of these $q$-series is 11-integral, and so they may be reduced modulo 11 to obtain
\begin{displaymath}
\begin{split}
&\widehat{F}_0(q):=F_0(q)\!\!\!\!\pmod{11}\equiv 2q^3+6q^4+7q^5+8q^6+5q^7+2q^8+2q^9+\dots\,\,\,\,\,\, \pmod{11},\\
&\widehat{F}_2(q):=F_2(q)\!\!\!\!\pmod{11}\equiv 6q^3+7q^4+10q^5+7q^6+8q^7+7q^8+6q^9+\dots \,\,\, \pmod{11},\\
&\widehat{F}_4(q):=F_4(q)\!\!\!\!\pmod{11}\equiv 7q^3+10q^4+8q^5+2q^6+4q^8+7q^9+\dots\,\,\,\,\,\,\, \ \ \ \ \ \ \pmod{11},\\
&\widehat{F}_6(q):=F_6(q)\!\!\!\!\pmod{11}\equiv 10q^3+8q^4+9q^5+10q^6+q^7+4q^8+10q^9+\dots\pmod{11},\\
&\widehat{F}_8(q):=F_8(q)\!\!\!\!\pmod{11}\equiv 8q^3+2q^4+6q^5+6q^6+8q^7+5q^8+8q^9+\dots \,\,\,\,\,\, \pmod{11}.
\end{split}
\end{displaymath}  
Using the congruences $E_2(q)\equiv E_{12}(q)\pmod{11}$ and $E_{10}(q)\equiv 1\pmod{11}$, we observe  that $\widehat{F}_0(q), \widehat{F}_2(q),\widehat{F}_4(q),\widehat{F}_6(q),$ and $\widehat{F_8}(q)$ are modular forms modulo $11$ of weight $120, 72, 84, 96,$ and $108$, respectively, on $\SL_2(\Z).$ 

We proceed to isolate the arithmetic progression of coefficients that is relevant for the theorem. 
We apply the differential operators to $\mathcal{U}_{10}(q)$ to eliminate  terms with exponents $n\equiv 0, 1,3,4,5,9\pmod{11}.$ 
The non-zero classes are the quadratic residues modulo 11. Using Fermat's Little Theorem and Euler's Criterion, this is achieved by
\begin{displaymath}
\begin{split}
G_1(q):&\equiv\sum_{n\equiv 2,6,7,8,10\pmod{11}}MO(10;n)q^n
\equiv \sum_{i=0}^4 -5[\Theta^{10}(\widehat{F}_{2i}(q))-\Theta^5(\widehat{F}_{2i}(q))]\pmod{11}.
\end{split}
\end{displaymath}
Next, we proceed to remove the terms with exponents that are quadratic non-residues apart from those with $n\equiv 7\pmod{11}$. For instance,
to eliminate $n\equiv 2\pmod{11}$ from $G_1(q)$ compute
\begin{displaymath}
\begin{split}
G_2(q):=\Theta (G_1(q))-2G_1(q)
&\equiv\sum_{n\equiv 6,7,8,10\pmod{11}}MO(10;n)q^n \pmod{11}.
\end{split}
\end{displaymath}
We repeat this process to remove terms with exponents $n\equiv 6,8,10\pmod{11},$ and
we get
\begin{displaymath}
\begin{split}
\sum_{n\equiv 7\pmod{11}}&MO(10;n)q^n \pmod{11} \\
&\equiv -5(\Theta^{4}(\widehat{F}_2)-\Theta^{9}(\widehat{F}_2))+9(\Theta^{3}(\widehat{F}_4)-\Theta^{8}(\widehat{F}_4))-3(\Theta^{2}(\widehat{F}_6)-\Theta^{7}(\widehat{F}_6))+\Theta(\widehat{F}_8)\\
&\ \ \ \ -\Theta^{6}(\widehat{F}_8)-4(\widehat{F}_0-\Theta^{5}(\widehat{F}_0))-5(\Theta^{4}(\widehat{F}_4)-\Theta^{9}(\widehat{F}_4))+9(\Theta^{3}(\widehat{F}_6)-\Theta^{8}(\widehat{F}_6))\\
&\ \ \ \ -3(\Theta^{2}(\widehat{F}_8)-\Theta^{7}(\widehat{F}_8))+\Theta(\widehat{F}_0)-\Theta^{6}(\widehat{F}_0)-4(\widehat{F}_2-\Theta^{5}(\widehat{F}_2))
-5(\Theta^{4}(\widehat{F}_6)-\Theta^{9}(\widehat{F}_6)) \\
&\ \ \ \ +9(\Theta^{3}(\widehat{F}_8)-\Theta^{8}(\widehat{F}_8))-3(\Theta^{2}(\widehat{F}_0)-\Theta^{7}(\widehat{F}_0))+\Theta(\widehat{F}_2)-\Theta^{6}(\widehat{F}_2)
-4(\widehat{F}_4-\Theta^{5}(\widehat{F}_4)) \\
&\ \ \ \ -5(\Theta^{4}(\widehat{F}_8)-\Theta^{9}(\widehat{F}_8))+9(\Theta^{3}(\widehat{F}_0)-\Theta^{8}(\widehat{F}_0))-3(\Theta^{2}(\widehat{F}_2)
-\Theta^{7}(\widehat{F}_2))+\Theta(\widehat{F}_4) \\
&\ \ \ \ -\Theta^{6}(\widehat{F}_4) 4(\widehat{F}_6-\Theta^{5}(\widehat{F}_6))-5(\Theta^{4}(\widehat{F}_0)
-\Theta^{9}(\widehat{F}_0)) +9(\Theta^{3}(\widehat{F}_2)-\Theta^{8}(\widehat{F}_2)) \\
&\ \ \ \ -3(\Theta^{2}(\widehat{F}_4)-\Theta^{7}(\widehat{F}_4))+\Theta(\widehat{F}_6)-\Theta^{6}(\widehat{F}_6)-4(\widehat{F}_8-\Theta^{5}(\widehat{F}_8)).
\end{split}
\end{displaymath}
Now we collect these terms so that
$$
\sum_{n\equiv 7\pmod{11}}MO(10;n)q^n \pmod{11}=Y_0(q)+Y_2(q)+Y_4(q)+Y_6(q)+Y_8(q),
$$
where $Y_{2i}(q)$ consists of  those  $\Theta^t(\widehat{F}_j(q))$ with weight congruent to $i$ modulo $10$. 
By Lemma~\ref{lem1}, we find that $Y_0(q),Y_2(q),Y_4(q),Y_6(q)$, and $Y_8(q)$ are modular forms modulo 11 with weights $228,180,192,204 $, and $216$, respectively on $\SL_2(\Z)$. Finally, by Sturm's Theorem \ref{Sturm} (i.e. checking at most $20$ terms), we find that each of these modular forms vanishes modulo 11, which implies the theorem.

\subsection{Proof of Theorem~\ref{Gordon}}

The generating function for the 3-colored partition function satisfies 
$$
P_3(q)=\sum_{n\geq0} c_3(n)q^n=\prod_{n\geq1}\frac{1}{(1-q^n)^3}\equiv \sum_{n\geq 0}p(n)q^{3n}\pmod 3.
$$
Therefore, we have that $3\mid c_3(n)$ whenever $3\nmid n.$
Furthermore, Gordon \cite{Gordon} proved that
\begin{displaymath}
c_3(11n+7)\equiv  0\pmod{11}.
\end{displaymath}
We further claim that $c_3(17n+15)\equiv 0\pmod{17}$. To prove this congruence, we employ Ramanujan's weight 12 cusp form $\Delta(z):=\eta(z)^{24}$ through the following observation:
$$
q^2\sum_{n\geq 0}c_3(n)q^n \cdot \prod_{n\geq1}(1-q^{17n})^3 \equiv q^2\prod_{n\geq1}(1-q^n)^{48}\pmod{17}=\Delta(z)^2.
$$
One easily checks that $\Delta^2 \ | \  T_{17}\equiv 0\pmod{17},$ which means that every seventeenth coefficient of $\Delta(z)^2$ vanishes modulo 17. This congruence follows immediately from the fact
that
$$
\prod_{n\geq1}(1-q^{17n})^3\in (\Z/17\Z)[q^{17}].
$$

Now suppose that $\ell\in \{3, 11, 17\}$ and
$1\leq a \equiv \ell-1\pmod{\ell}.$ 
If $\ell\nmid n(n+1)/2,$ then we have
$$
\frac{(2n+1)(n+a)!}{(2a+1)! (n-a)!}\equiv 0\pmod{\ell}.
$$
Therefore, the identity in Lemma~\ref{newgen_for_U} collapses modulo $\ell$ and gives
$$\mathcal{U}_a(q)=\sum_{n\geq 0}MO(a;n)q^n\equiv  P_3(q)\cdot \sum_{n\geq 0} A(a,\ell; \ell n)q^{\ell n}\pmod \ell.
$$
The power series on the right is in $(\Z/\ell\Z)[q^{\ell}],$ and so the $MO(a;n)$ inherit the $c_3(n)$ congruences.

\section{Proof of Theorem~\ref{Lacunarity}}\label{Section5}

\noindent
Here we prove Theorem~\ref{Lacunarity} and Corollary~\ref{natural}.

\subsection{Nuts and Bolts}
We first recall the definition of Hecke operators.
Let $m$ be a positive integer and $f(z) = \sum_{n=0}^{\infty} a(n)q^n \in M_{k}   $. Then the action of Hecke operator $T_m$ on $f(z)$ is defined by 
\begin{align}
f(z)\,|\,T_m := \sum_{n=0}^{\infty} \left(\sum_{d\mid \gcd(n,m)}d^{k-1}a\left(\frac{nm}{d^2}\right)\right)q^n.
\end{align}
In particular, if $m=p$ is a prime, we have 
\begin{align}
f(z) \,|\, T_p :=f(z)\,|\, U_p+p^{k-1}f(z)\,|\, V_p,
\end{align}
where $f(z)\,|\, U_p:=\sum_{n=0}^{\infty}a(pn)q^{n}$ and $f(z)\,|\, V_p:=\sum_{n=0}^{\infty}a(n)q^{pn}$.

Let's recall a result of Serre \cite{Serre2} (also see \cite[Lemma 2.63 and Theorem 2.65]{CBMS}) on the action of Hecke operator on cusp forms. For a number field $K$, let $\mathcal{O}_{K}$ denote its ring of integers.
\begin{lemma}\label{Serre_lemma}
For $1\leq i\leq t$, let  $f_i(z)=\sum_{n=1}^{\infty}a_i(n)q^n\in M_k$ be a modular form
with coefficients in the ring of integers of a number field $O_K.$ Then the following are true.
	
	\noindent	
	(i) If $\fm\subset\mathcal{O}_{K}$ is an ideal of norm $M$, then a positive proportion of the primes $p\equiv-1\pmod{M}$ satisfy
	$$f_1(z) \mid T_{p} \equiv f_2(z) \mid T_{p} \equiv\cdots f_t(z) \mid T_{p} \equiv 0 \pmod{\fm}.$$
	(ii) There is a constant $a>0$  such that for every $1\leq i\leq t$ we have
	$$ \# \left\{n\leq X: a_i(n)\not\equiv 0 \pmod{\fm} \right\}= O\left(X/(\log{}X)^{a}\right).$$
	\end{lemma}

\noindent
We next recall some facts about $p$-adic modular forms developed by Serre \cite{Serre3}.
Let $p$ be a prime. Consider the field of $p$-adic numbers $\Q_p$, with its non-archimedean valuation $\nu_p$. We say $x\in \Q_p$ is $p$-integral if $\nu_p(x)\geq 0$. Let $f =\sum a(n) q^n\in\Q_p[[q]]$ be a formal power series, we define $\nu_p(f) := \inf_{n} \nu_p(a_n)$. If $\nu_p(f)\geq m$, we write as well $f\equiv 0\pmod{p^m}$. Assume $\{f_i\}$ to be a sequence of elements in $\Q_p[[q]]$. We say that $f_i\to f$ if the coefficients of $f_i$ tend uniformly to those of $f$, i.e., $\nu_p(f-f_i)\to\infty$.
A $p$-{\it adic modular form} $f$ is a formal series with coefficients in $\Q_p$ which is the
	limit of classical modular forms $f_i$ of weights $k_i$.

In order to prove Theorem \ref{Lacunarity}, we need the following preliminary result.

\begin{lemma}\label{lem2}
	The following are true:
	
\noindent
	(i) If $m$ is a positive integer, then we have that
	$$
	E_2(z)\equiv \frac{1}{(2^m-1)}\sum_{i=1}^{m}2^{i-1} E_{2+3\cdot2^{m+1}}(z) \,|\, V_{2^{i-1}}\pmod{2^m}.
	$$
	Moreover, $E_2(z)\pmod{2^m}$ is the reduction of a weight $2+3\cdot2^{m+1}$ modular form on $\SL_2(\Z)$.\\
	(ii) If $m$ is a positive integer, then we have that
	$$
	E_2(z)\equiv\frac{2}{(3^m-1)}\sum_{i=1}^{m}3^{i-1} E_{2+4\cdot3^{m}}(z) \,|\, V_{3^{i-1}}\pmod{3^m}.
	$$
	Moreover, $E_2(z)\pmod{3^m}$ is the reduction of a weight $2+4\cdot3^{m}$ modular form on $\SL_2(\Z)$.\\
	(iii)  If $p\geq 5$ is prime and $m$ is a positive integer, then we have that
	$$
	E_2(z) \equiv \frac{(p-1)}{(p^m-1)}\sum_{i=1}^{m}p^{i-1}E_{2+(p-1)p^{m-1}}(z) \,|\, V_{p^{m-1}}\pmod {p^m}.
	$$
	In particular, $E_2(z)\pmod{p^m}$ is the reduction of a weight $2+p^{m-1}(p-1)$ modular form on $\SL_2(\Z)$.
\end{lemma}
\begin{proof}
	\noindent
	Let $g(z)=\sum_{n=0}^{\infty} b(n) q^n$ be a weight $k$ modular form on $\SL_2(\Z)$ with $p$-integral coefficients. Then $g(z) \,|\, T_p = \sum_{n=0}^{\infty} b(pn)q^n+p^{k-1}b(n) q^{pn}\in M_k$. Since $E_{p^{m-1}(p-1)}(z)\equiv 1\pmod{p^m}$, we have that $\{g(z)E_{p^{m-1}(p-1)}(z)\}$ converges to $g(z)$ $p$-adically. Hence, $g(z)$ is a $p$-adic modular form. Also, we have  the convergence 
$$g(z)E_{p^{m-1}(p-1)}(z) \,|\, T_p\longrightarrow g(z) \,|\, U_p=\sum_{n=0}^{\infty} b(pn)q^n.$$ 
Hence, $U_p$ is an operator on $M_k$ and so is $V_{p}$, defined as $g(z) \,|\, V_{p}=p^{1-k}(g(z) \,|\, T_p-g(z) \,|\,U_p)$. Now, our proof of the lemma follows from \cite[Example on Page 210]{Serre3}.	
\end{proof}

\subsection{Proof of Theorem \ref{Lacunarity}}
By Theorem \ref{ExplicitExpressionUa} and \ref{ExplicitExpressionUaStar}, we have 
\begin{align}\label{eq2}
\mathcal{U}_a(q)&= \sum_{t=0}^a w_t(a)
\sum_{\substack {\alpha, \beta, \gamma\geq 0\\ \alpha+2\beta+3\gamma=t}}
c(\alpha,\beta,\gamma) E_2(q)^{\alpha} E_4(q)^{\beta} E_6(q)^{\gamma}=F_0(q)+F_2(q)+\cdots+F_{2a}(q),\nonumber\\
\mathcal{U}_a^{\star}(q)&=\sum_{t=0}^a w_t^{\star}(a)\cdot \mathbb{E}_{2t}^{\star}(q)=F_0^{\star}(q)+F_2^{\star}(q)+\cdots+F_{2a}^{\star}(q),
\end{align}
where $F_{2i}(q)$ and $F_{2i}^{\star}(q)$, for $0\leq i\leq a$, are quasimodular forms of weight $2i$ on $\SL_2(\Z)$. Using Lemma \ref{lem2} and the Chinese Remainder Theorem, we find that $F_{2i}(q)$ and $F_{2i}^{\star}(q)$, for $0\leq i\leq a$, are modular forms modulo any integer $m$ on $\SL_2(\Z)$.
Employing Lemma~\ref{Serre_lemma} $(i)$ on \eqref{eq2}, we complete the proof of first and second parts of Theorem~\ref{Lacunarity}  and finally applying Lemma~\ref{Serre_lemma} $(ii)$ to \eqref{eq2}, claim $(iii)$ follows.

\subsection{Proof of Corollary~\ref{natural}}
	By Theorem \ref{ExplicitExpressionUa}, we find that
	\begin{align*}
	\mathcal{U}_2(q)&=\frac1{2^3}\sum_{n\geq0}[(-2n + 1)\sigma_1(n)+\sigma_3(n)]q^n, \\
	\mathcal{U}_3(q)&=\frac1{2^7\cdot 3\cdot 5}\sum_{n\geq0}[(40n^2 - 100n + 37)\sigma_1(n) + (-30n + 50)\sigma_3(n) + 3\sigma_5(n)]q^n, \\
	\mathcal{U}_4(q)&=\frac{1}{2^{10}3^3\cdot 5\cdot 7}\sum_{n\geq 0}\big[(-840n^3+5880n^2-9870n+3229)\sigma_1  \nonumber \\
	&\quad \quad \quad +(756n^2-4410n+4935)\sigma_3+(-126n+441)\sigma_5+5\sigma_7\big]q^n, \\
	\mathcal{U}_5(q)&=\frac1{2^{15}3^35^2\cdot 7}\sum_{n\geq0}[(3360n^4 - 50400n^3 + 223440n^2 - 314200n + 96111)\sigma_1(n) \\
	& \quad \quad \quad + (-3360n^3 + 45360n^2 - 167580n + 157100)\sigma_3(n) \\
	& \quad \quad \quad + (720n^2 - 7560n + 16758)\sigma_5(n) + (-50n + 300)\sigma_7(n) + \sigma_9(n)]q^n.
	\end{align*}
	Let $s$ and $t$ be non-negative integers. If $k$ is a positive odd, then for primes $p\equiv -1\pmod{s}$ we have
	\begin{align}\label{eq1}
	(pn)^t\sigma_k(pn)=(pn)^t\sigma_k(n)\sigma_k(p)=(pn)^t\sigma_k(n)(1+p^k)\equiv 0\pmod{s}
	\end{align}
	for all $n$ coprime to $p$. 
	Corollary~\ref{natural} follows by applying (\ref{eq1}) appropriately in each case.


\begin{thebibliography}{99}

\bibitem{AAT} T. ~Amdeberhan, G. E. ~Andrews, R. ~Tauraso,
\emph{ Extensions of MacMahon's sums of divisors},
preprint available at {\tt https://arxiv.org/pdf/2309.03191.pdf} (2023).

\bibitem{Andrews-Rose} G. E. ~Andrews, S. C. F. ~Rose,
\emph{MacMahon's sum-of-divisors functions, Chebyshev polynomials, and quasimodular forms}, 
J. Reine Angew. Math. \textbf{676} (2013), 97--103.

\bibitem{Brindle} B. ~Brindle,
\emph{A unified approch to qMZVs},
preprint available at {\tt https://arxiv.org/pdf/2111.00051.pdf} (2021).

\bibitem{Gordon} B. Gordon, \emph{Ramanujan congruences for $p_{-k}\pmod{11^r}$}, Glasgow J. Math.
\textbf{24} (1983), 107-123.

\bibitem{Han} G.-H. ~Han,
\emph{The Nekrasov-Okounkov hook length formula: refinement, elementary proof, extension and applications},
Annales de l'Institut Fourier, Volume 60 (2010) no. 1, 1-29.

\bibitem{Zagier} M. ~Kaneko, D. ~Zagier,
\emph{A generalized Jacobi theta function and quasimodular forms},
The moduli space of curves (Texas Island, 1994), Progr. Math., vol. \textbf{129}, Birkh\"auser Boston, Boston, MA, 1995, 165--172.

\bibitem{MacMahon} P. A. ~MacMahon, 
\emph{Divisors of Numbers and their Continuations in the Theory of Partitions},
Proc. London Math. Soc. (2) \textbf{19} (1920), no.1, 75-113 
[also in Percy Alexander MacMahon Collected Papers, Vol.2, pp. 303--341 (ed. G.E. Andrews), MIT Press, Cambridge, 1986].

\bibitem{NekOk} N. A.~Nekrasov, A.~Okounkov, 
\emph{Seiberg-Witten theory and random partitions, in The unity of mathematics},
Progress in Mathematics, Birkh\"auser Boston, 2006, vol. 244, 525-596.

 \bibitem{CBMS}  K. Ono,
 \emph{The web of modularity: arithmetic of the coefficients of modular forms and $q$-series}, CBMS Regional Conference Series in Mathematics, \textbf{102}, Amer. Math. Soc., Providence, RI, 2004.

\bibitem{Polya} G. P\'olya,
\emph{Kombinatorische Anzahlbestimmungen f\"ur Gruppen, Graphen und chemische Verbindungen},
 Acta Mathematica, \textbf{68} (0), (1937), 145-254.

\bibitem{Rama} S. ~Ramanujan, 
\emph{On certain arithmetical functions}, 
Trans. Camb. Phil. Soc., \textbf{22} (1916), 159--184.

\bibitem{rose} S. C. F. ~Rose,
\emph{Quasimodularity of generalized sum-of-divisors functions},
Research in Number Theory \textbf{1} (2015), Paper No. 18, 11 pp.

\bibitem{Serre1} J.-P. Serre,
\emph{A course in arithmetic}, Springer-Verlag, New York, 1973.

\bibitem{Serre2} J.-P. Serre, 
\emph{ Divisibilit\'{e} de certaines fonctions arithm\'{e}tiques}, L'Enseignement Math. \textbf{22} (1976), 227--260.

\bibitem{Serre3} J.-P. Serre,
\emph{ Formes modulaires et fonctions z\^{e}ta $p$-adiques}, Springer Lect. Notes \textbf{350} (1973), 191--268.

\bibitem{St} J. Sturm,
\emph{On the congruence of modular forms}, 
Springer Lect. Notes in Math. \textbf{1240} (1984), 275--280.

\bibitem{SwD} H. P. F. Swinnerton-Dyer, 
\emph{On $\ell$-adic representations and congruences for coefficients of
	modular forms}, Springer Lect. Notes. \textbf{350} (1973), 1--55.

\end{thebibliography}
\end{document}